\documentclass[12pt,reqno]{amsart}

\usepackage{amssymb}

\def\CC{{\mathbb C}} 

\newtheorem{theorem}{Theorem}[section]
\newtheorem{lemma}[theorem]{Lemma}
\newtheorem{corollary}[theorem]{Corollary}

\theoremstyle{definition}

\numberwithin{equation}{section}

\begin{document}
\baselineskip=15.5pt

\title[Symplectic structure on moduli of projective structures]{On the symplectic
structure over the moduli space of projective structures on a surface}

\author[I. Biswas]{Indranil Biswas}

\address{Mathematics Department, Shiv Nadar University, NH91, Tehsil Dadri,
Greater Noida, Uttar Pradesh 201314, India}

\email{indranil.biswas@snu.edu.in, indranil29@gmail.com}



\date{}

\begin{abstract}
The moduli space of projective structures on a compact oriented surface $\Sigma$ has a holomorphic symplectic
structure, which is constructed by pulling back, using the monodromy map, the Atiyah--Bott--Goldman
symplectic form on the character variety ${\rm Hom}(\pi_1(\Sigma), \, 
\text{PSL}(2, {\mathbb C}))/\!\!/\text{PSL}(2, {\mathbb C})$.
We produce another construction of this symplectic form. This construction shows that the symplectic form
on the moduli space is actually algebraic. Note that the monodromy map is only holomorphic and \textit{not}
algebraic, so the first construction does not give abgebraicity of the pulled back form.
\end{abstract}

\maketitle

\section{Introduction}

Let $\Sigma$ be a compact oriented $C^\infty$ surface of genus $g$. The corresponding Teichm\"uller space
${\rm Teich}(\Sigma)$ is the quotient of ${\rm Com}(\Sigma)$, the space of $C^\infty$ complex structures on
$\Sigma$ compatible with its orientation, by the natural action of the group ${\rm Diff}^0(\Sigma)$ given by
the connected component, containing the identity element, of the group ${\rm Diff}^+(\Sigma)$ defined by
the orientation preserving diffeomorphisms of $\Sigma$. The mapping class group
${\rm Mod}(\Sigma)\,:=\, {\rm Diff}^+(\Sigma)/{\rm Diff}^0(\Sigma)$ acts on ${\rm Teich}(\Sigma)$, and
the quotient ${\mathcal M}_g\,:=\, {\rm Teich}(\Sigma)/{\rm Mod}(\Sigma)$ is the moduli space of smooth
complex projective curves of genus $g$.

A projective structure on $\Sigma$ is defined by giving a coordinate atlas, where each coordinate function 
has target ${\mathbb C}{\mathbb P}^1$, such that all the transition functions are M\"obius transformations 
(same as holomorphic automorphisms of ${\mathbb C}{\mathbb P}^1$). Let $\text{Proj}(\Sigma)$ denote the space
of $C^\infty$ projective structures on $\Sigma$ compatible with its orientation. A projective structure
on $\Sigma$ produces a flat $\text{PSL}(2, {\mathbb C})$--bundle on $\Sigma$ together with a (non-flat)
reduction of structure group to a Borel subgroup of $\text{PSL}(2, {\mathbb C})$. 
Sending a projective structure to the monodromy homomorphism of the corresponding flat
$\text{PSL}(2, {\mathbb C})$--bundle, we obtain a holomorphic map from
${\mathcal P}(\Sigma)\,\ :=\,\, {\rm Proj}(\Sigma)/{\rm Diff}^0(\Sigma)$ to the
$\text{PSL}(2, {\mathbb C})$--character variety for $\Sigma$. This map is a local biholomorphism,
and hence the pullback of the Atiyah--Bott--Goldman symplectic form on the character variety gives
a holomorphic symplectic form on ${\mathcal P}(\Sigma)$. The symplectic form on the character variety
if preserved by the action of ${\rm Mod}(\Sigma)$, and also the monodromy map is ${\rm Mod}(\Sigma)$--equivariant.
Therefore, the above holomorphic symplectic form on ${\mathcal P}(\Sigma)$ descends to a holomorphic
symplectic form $\Theta_g$ on ${\mathcal P}(\Sigma)_g\,:=\, {\mathcal P}(\Sigma)/{\rm Mod}(\Sigma)$.

Now, ${\mathcal P}(\Sigma)_g$ is an algebraic orbifold; in fact, it an algebraic torsor over
${\mathcal M}_g$ for the cotangent bundle $\Omega_{{\mathcal M}_g}$. The above monodromy map is
\textit{not} algebraic because the degree of the map is infinite. 

Our aim here is to give an alternative construction of the symplectic form $\Theta_g$ on
${\mathcal P}(\Sigma)_g$. Our construction shows that the form $\Theta_g$ is actually algebraic.

\section{A parameter space for projective structures}\label{se4.1}

\subsection{Projective structure}

A holomorphic coordinate chart on a connected Riemann surface $X$ is a pair
$(U,\, \phi)$, where $U\, \subset\, X$ is an open subset and $\phi\,
:\, U\, \longrightarrow\, {\mathbb C}{\mathbb P}^1$ is a holomorphic embedding. A holomorphic
coordinate atlas on $X$ is a collection of coordinate charts $\{(U_i,\, \phi_i)\}_{i\in I}$
such that $X\,=\, \bigcup_{i\in I} U_i$. A projective structure on $X$ is given by a holomorphic
coordinate atlas $\{(U_i,\, \varphi_i)\}_{i\in I}$ such that
for every $i,\, j\, \in\, I\times I$ and every nonempty connected
component $V_c\, \subset\, U_i\bigcap U_j$, there is an element
$A^c_{j,i} \, \in\, \text{Aut}({\mathbb C}{\mathbb P}^1)
\,=\, \text{PSL}(2, {\mathbb C})$ whose restriction to $\phi_i(V_c)$ coincides
with $(\phi_j\circ\phi^{-1}_i)\big\vert_{\phi_i(V_c)}$.
Two holomorphic coordinate atlases $\{(U_i,\, \phi_i)\}_{i\in I}$ and $\{(U_i,\, \phi_i)\}_{i\in I'}$
satisfying the above condition are called {\it equivalent} if their union
$\{(U_i,\, \phi_i)\}_{i\in I\cup
I'}$ also satisfies the above condition. A {\it projective structure} on $X$ is an equivalence class of
holomorphic coordinate atlases satisfying the above condition (see \cite{Gu}).

Fix the Borel subgroup
\begin{equation}\label{bsg}
B\, :=\, \Big\{
\begin{pmatrix}
a & b\\
c &d
\end{pmatrix}
\, \in\, \text{PSL}(2,{\mathbb C})\, \mid\, b\, =\, 0\Big\}\, \subset\, \text{PSL}(2,{\mathbb C})\, .
\end{equation}
Let $P\, \longrightarrow\, X$ be a holomorphic principal $\text{PSL}(2,{\mathbb C})$--bundle on $X$
equipped with a holomorphic connection $D$. Let $P_B\, \subset\, P$
be a holomorphic reduction of structure group, to the subgroup $B$ in \eqref{bsg}, given by a
holomorphic section
\begin{equation}\label{e1}
\sigma\, :\, X\, \longrightarrow\, P/B
\end{equation}
of the natural projection $f\, :\, P/B\, \longrightarrow\, X$. The connection $D$ on $P$ induces a connection on the
fiber bundle $f$, which, in turn, decomposes the tangent bundle $T(P/B)$ as
$T(P/B)\,=\, T_f\oplus {\mathcal H}$,
where $T_f\, \subset\, T(P/B)$ is the relative tangent bundle for the projection $f$ and ${\mathcal H}\, \subset\,
T(P/B)$ is the horizontal subbundle for the connection on $P/B$. Let
\begin{equation}\label{e2}
\widetilde{\sigma}\, :\, TX\, \longrightarrow\, \sigma^*T_f
\end{equation}
be the composition of homomorphisms
$$
TX\, \stackrel{d\sigma}{\longrightarrow}\,\sigma^* T(P/B)\,=\, \sigma^*T_f\oplus \sigma^*{\mathcal H}
\, \longrightarrow\, \sigma^*T_f
$$
constructed using the above decomposition of $T(P/B)$, where $d\sigma$ is the differential of $\sigma$ in
\eqref{e1}.

A projective structure on $X$ is a triple $(P,\, P_B,\, D)$ as above such that the homomorphism
$\widetilde{\sigma}$ in \eqref{e2} is an isomorphism \cite{Gu}.

The holomorphic cotangent bundle of $X$ will be denoted by $K_X$. The extensions of ${\mathcal O}_X$ by $K_X$ are
parametrized by $$H^1(X,\, \text{Hom}({\mathcal O}_X,\, K_X))\,=\, H^1(X,\, K_X)\,=\, \mathbb C.$$
Let
\begin{equation}\label{z1}
0\, \longrightarrow\, K_X\, \longrightarrow\, V \, \longrightarrow\, {\mathcal O}_X
\, \longrightarrow\, 0
\end{equation}
be the extension corresponding to $1\, \in\, H^1(X,\, \text{Hom}({\mathcal O}_X,\, K_X))\,=\, \mathbb C$.

Let
\begin{equation}\label{e3}
P^0\, \longrightarrow\, X
\end{equation}
be the principal $\text{PSL}(2,{\mathbb C})$--bundle given by ${\mathbb P}(V)$; the fiber
of $P^0$ over any $x\, \in\, X$ is the space of all holomorphic isomorphisms ${\mathbb C}{\mathbb P}^1\,
\longrightarrow\, P^0_x$. The line subbundle $K_X\, \hookrightarrow\, V$ in \eqref{z1} produces a reduction
of structure group
\begin{equation}\label{f1}
P^0_B\, \subset\, P^0
\end{equation}
to the Borel subgroup $B\,\subset\, \text{PSL}(2,{\mathbb C})$ in \eqref{bsg}. The fiber of $P^0_B$ over any
$x\, \in\, X$ is the space of all isomorphisms ${\mathbb C}{\mathbb P}^1\,
\longrightarrow\, P^0_x$ that take the point ${\mathbb C}\cdot e_2\, \in\, {\mathbb C}{\mathbb P}^1$ to the
point of $P^0_x$ given by the line $(K_X)_x\, \subset\, V_x$, where $\{e_1,\, e_2\}$ is the
standard basis of ${\mathbb C}^2$.

For any projective structure $(P,\, P_B,\, D)$ on $X$, the pair $(P,\, P_B)$ coincides with $(P^0,\, P^0_B)$ 
\cite{Gu}. Furthermore, for every connection $D$ on $P^0$, the triple $(P^0,\, P^0_B,\, D)$ is actually a 
projective structure on $X$. From these it follows that the space of all projective structures on $X$ is an 
affine space modeled on the vector space $H^0(X,\, K^{\otimes 2}_X)$.

\subsection{A symplectic form}

Fix a compact oriented $C^\infty$ surface $\Sigma$ of genus $g\,\ge\, 2$.
Denote by ${\rm Diff}^+(\Sigma)$ the group
of all orientation preserving diffeomorphisms of $\Sigma$. Let ${\rm Diff}^0(\Sigma)\, \subset\,
{\rm Diff}^+(\Sigma)$ be the
connected component containing the identity element. Then the quotient
$$
{\rm Mod}(\Sigma)\,:=\,\pi_0({\rm Diff}^+(\Sigma))\,=\, {\rm Diff}^+(\Sigma)/{\rm Diff}^0(\Sigma)
$$
is known as the mapping class group. Let ${\rm Com}(\Sigma)$ denote the space of $C^\infty$ complex
structures on $\Sigma$ compatible with its orientation. The quotient ${\rm Teich}(\Sigma)\,:=\,
{\rm Com}(\Sigma)/{\rm Diff}^0(\Sigma)$ for the natural action of ${\rm Diff}^0(\Sigma)$ on
${\rm Com}(\Sigma)$ is the Teichm\"uller space for $\Sigma$. The
group ${\rm Mod}(\Sigma)$ acts property discontinuously on ${\rm Teich}(\Sigma)$. The corresponding orbifold
$$
{\mathcal M}_g\,=\, {\rm Teich}(\Sigma)/{\rm Mod}(\Sigma)
$$
is the moduli space of smooth complex projective curves of genus $g$.

Let ${\rm Proj}(\Sigma)$ denote the space of all projective structures on the $C^\infty$ surface $\Sigma$
compatible with its orientation. The group ${\rm Diff}^+(\Sigma)$ has a natural action on it. Consider the
quotient
$$
{\mathcal P}(\Sigma)\,\ :=\,\, {\rm Proj}(\Sigma)/{\rm Diff}^0(\Sigma).
$$
The natural map ${\rm Proj}(\Sigma)\, \longrightarrow\, {\rm Com}(\Sigma)$ produces a projection
\begin{equation}\label{vp}
\varpi\,\, :\,\, {\mathcal P}(\Sigma)\,\,\longrightarrow\,\, {\rm Teich}(\Sigma)
\end{equation}
The action of the group ${\rm Diff}^+(\Sigma)$ on ${\rm Proj}(\Sigma)$ produces an action of
${\rm Mod}(\Sigma)$ acts on ${\mathcal P}(\Sigma)$. Let
\begin{equation}\label{v2}
{\mathcal P}_g\,\,:=\,\, {\mathcal P}(\Sigma)/{\rm Mod}(\Sigma)
\end{equation}
be the corresponding orbifold. The map $\varpi$ in \eqref{vp} produces a map
\begin{equation}\label{vp2}
\varpi'\, :\, {\mathcal P}_g\, =\, {\mathcal P}(\Sigma)/{\rm Mod}(\Sigma)\,\longrightarrow\,{\rm Teich}(\Sigma)/{\rm Mod}(\Sigma)
\,=\, {\mathcal M}_g.
\end{equation}

There is a universal family ${\mathcal X}_{{\rm Teich}(\Sigma)}/{\rm Teich}(\Sigma)$ of Riemann surfaces
on ${\rm Teich}(\Sigma)$.
The fiber of the map $\varpi$ in \eqref{vp} over any $t\, \in\, {\rm Teich}(\Sigma)$ is an affine space for
the space all quadratic differentials on the Riemann surface $({\mathcal X}_{{\rm Teich}(\Sigma)})_t$. In fact,
${\mathcal P}(\Sigma)$ is a holomorphic torsor over ${\rm Teich}(\Sigma)$ for the holomorphic cotangent
bundle $\Omega_{{\rm Teich}(\Sigma)}$ of ${\rm Teich}(\Sigma)$.

We note that ${\mathcal P}_g$ has a natural algebraic structure such that the map $\varpi'$ in \eqref{vp2}
is algebraic. In fact, ${\mathcal P}_g$ is an algebraic torsor over ${\mathcal M}_g$ for the
cotangent bundle $\Omega_{{\mathcal M}_g}$ of ${\mathcal M}_g$.

Given $x\,\in \,\Sigma$, the fundamental group $\pi_1(\Sigma,\, x)$ has $2g$ generators that 
are subject to a single relation. So the space of all homomorphisms from $\pi_1(\Sigma, \,x)$ to ${\rm PSL}
(2,\CC)$ that lift to a homomorphism from $\pi_1(\Sigma, \,x)$ to ${\rm SL}(2,\CC)$ constitute a Zariski closed
subset of ${\rm PSL}(2,\CC)^{2g}$. This Zariski closed subset
is evidently affine. The action of ${\rm PSL}(2, \CC)$ by conjugation
on this affine space of homomorphisms is also in 
the affine category. The homomorphisms that are irreducible in the sense that their image in ${\rm PSL}(2,\CC)$ is 
Zariski dense, and lift to a homomorphism from $\pi_1(\Sigma, \,x)$ to ${\rm SL}(2,\CC)$,
make up an open-dense subset ${\mathcal R}'(\Sigma)$. It has the property that ${\rm PSL}(2,\CC)$ acts freely on 
it. In fact, we can form this orbit variety in the affine category. This affine variety is called the 
\emph{character variety} of $\Sigma$ and is denoted by ${\mathcal R}(\Sigma)$. This
${\mathcal R}(\Sigma)$ is smooth of dimension $6(g-1)$. Since 
the inner automorphisms of $\pi_1(\Sigma,\, x)$ are killed by this procedure, the variety ${\mathcal R}(\Sigma)$ is 
actually independent of the choice of our base point $x\,\in\,\Sigma$ and (thus) comes with an
action of the mapping class group 
${\rm Mod}(\Sigma)$ (which acts on $\pi_1(\Sigma,\, x)$ by outer automorphisms). We can interpret
${\mathcal R}(\Sigma)$ as the moduli space of topologically trivial flat principal ${\rm PSL}(2,\CC)$-bundles
on $\Sigma$ with irreducible monodromy.

The character variety ${\mathcal R}(\Sigma)$ also comes with a closed nondegenerate $2$-form (to which we refer simply
as an \emph{algebraic symplectic 
structure}), \cite{AB}, \cite{Go}, that we shall denote by
\begin{equation}\label{e10}
\Theta_{\Sigma}\, \in\, H^0({\mathcal R}(\Sigma),\, \bigwedge\nolimits^2 T^*{\mathcal R}(\Sigma)).
\end{equation}
It is natural, and therefore it is preserved by the action of ${\rm Mod}(\Sigma)$. In
\cite{BL} an alternative construction of $\Theta_{\Sigma}$ is given.

Since a projective structure on $\Sigma$ determines a principal ${\rm PSL}(2,\CC)$-bundle with a
flat connection, we have a natural, ${\rm Mod}(\Sigma)$-equivariant monodromy map
\begin{equation}\label{e11}
\Psi\, :\, {\mathcal P}(\Sigma) \, \longrightarrow\, {\mathcal R}(\Sigma).
\end{equation}
This map $\Psi$ is known to be a local (holomorphic)
isomorphism \cite[p.~272, Theorem]{Hu}, \cite{He}, but its fibers are not finite (this follows from
\cite[p.~141, Theorem 6.2]{Ba} (see \cite[p.~120]{Ba}) and \cite{GKM}). The pullback
of the form in \eqref{e10}
$$
\Psi^*\Theta_{\Sigma}\, \in\, H^0({\mathcal P}(\Sigma),\, \bigwedge\nolimits^2 T^*{\mathcal P}(\Sigma))
$$
is a holomorphic symplectic form on ${\mathcal P}(\Sigma)$. Since $\Theta_{\Sigma}$
is preserved by the action of ${\rm Mod}(\Sigma)$, and $\Psi$ is
${\rm Mod}(\Sigma)$-equivariant, it follows that $\Psi^*\Theta_{\Sigma}$
descends to a holomorphic symplectic form on ${\mathcal P}_g$ (see \eqref{v2})
\begin{equation}\label{e12}
\Theta_g\,\, \in\,\, H^0({\mathcal P}_g,\, \bigwedge\nolimits^2 T^*{\mathcal P}_g).
\end{equation}

\section{Symplectic form on moduli of projective structures}

Recall that ${\mathcal P}_g$ has an algebraic structure.
We will construct an algebraic $2$--form on ${\mathcal P}_g$. For that we need to describe the
tangent bundle $T{\mathcal P}_g$.

\subsection{Infinitesimal deformations of projective structures}

Let $X$ be a smooth complex projective curve $X$ of genus $g\, \geq\, 2$, and let $\delta$ be a projective structure
on $X$. Then there is a holomorphic differential operator of order three
\begin{equation}\label{delta}
\Delta_\delta\, :\, TX\, \longrightarrow\, K^{\otimes 2}_X\, ;
\end{equation}
we will briefly recall its construction (see \cite[pp.~264--265]{Hu}). Since the space of projective structures on
an open subset $U\, \subset\, X$ is an affine space for $H^0(U,\, K^{\otimes 2}_U)$, given $\chi\,\in\,
H^0(U,\, TU)$, and any coordinate function $\phi$ on $U$ compatible with $\delta$, the Lie derivative
$L_\chi\phi$ of $\phi$ with respect to $\chi$ is an element of $H^0(U,\, K^{\otimes 2}_U)$. This element
\begin{equation}\label{dde}
L_\chi\phi\, \in\, H^0(U,\, K^{\otimes 2}_U)
\end{equation}
is independent of the choice of $\phi$. The differential operator
$\Delta_\delta$ in \eqref{delta} sends $\chi$ to $L_\chi\phi$.

We will describe an alternative construction of $\Delta_\delta$.
The natural isomorphism of $\text{Aut}({\mathbb C}{\mathbb P}^1)$ with $\text{PSL}(2,{\mathbb C})$ gives an
isomorphism
\begin{equation}\label{liso}
\text{sl}(2,{\mathbb C})\, \stackrel{\sim}{\longrightarrow}\,
H^0({\mathbb C}{\mathbb P}^1,\, T{\mathbb C}{\mathbb P}^1),
\end{equation}
where $\text{sl}(2,{\mathbb C})$ is the Lie algebra of $\text{PSL}(2,{\mathbb C})$,
that takes the Lie algebra operation on $\text{sl}(2,{\mathbb C})$ to the Lie bracket operation
of vector fields. For any $n\, \geq\, 0$, let
\begin{equation}\label{ee}
\eta_n\, :\, {\mathbb C}{\mathbb P}^1\times H^0({\mathbb C}{\mathbb P}^1,\, T{\mathbb C}{\mathbb P}^1)
\, \longrightarrow\, J^n(T{\mathbb C}{\mathbb P}^1)
\end{equation}
be the natural restriction map of global sections, where ${\mathbb C}{\mathbb P}^1\times
H^0({\mathbb C}{\mathbb P}^1,\, T{\mathbb C}{\mathbb P}^1)$ is the trivial vector bundle
on ${\mathbb C}{\mathbb P}^1$ with fiber
$H^0({\mathbb C}{\mathbb P}^1,\, T{\mathbb C}{\mathbb P}^1)$. The map $\eta_2$ is an isomorphism.
So the composition of homomorphisms $$\eta_3\circ(\eta_2)^{-1}\, :\, J^2(T{\mathbb C}{\mathbb P}^1)
\,\longrightarrow\, J^3(T{\mathbb C}{\mathbb P}^1)$$ has the following property: Consider the natural jet sequence
\begin{equation}\label{ee2}
0\,\longrightarrow\, K^{\otimes 2}_{{\mathbb C}{\mathbb P}^1} \,\stackrel{\alpha}{\longrightarrow}\,
J^3(T{\mathbb C}{\mathbb P}^1)\,\stackrel{\beta}{\longrightarrow}\, 
J^2(T{\mathbb C}{\mathbb P}^1)\,\longrightarrow\, 0\, .
\end{equation}
We have $\beta\circ (\eta_3\circ(\eta_2)^{-1})\,=\, {\rm Id}_{J^2(T{\mathbb C}{\mathbb P}^1)}$.
We note that $\eta_3\circ(\eta_2)^{-1}$ is equivariant for the actions of $\text{PSL}(2,{\mathbb C})$
on $J^3(T{\mathbb C}{\mathbb P}^1)$ and $J^2(T{\mathbb C}{\mathbb P}^1)$ given by the action
of $\text{PSL}(2,{\mathbb C})$ on ${\mathbb C}{\mathbb P}^1$.

Let
\begin{equation}\label{ee3}
\Delta_{{\mathbb P}^1}\, :\, J^3(T{\mathbb C}{\mathbb P}^1)\,\longrightarrow\,
K^{\otimes 2}_{{\mathbb C}{\mathbb P}^1}
\end{equation}
be the unique homomorphism such that $\text{kernel}(\Delta_{{\mathbb P}^1})\,=\, 
\eta_3((\eta_2)^{-1}(J^2(T{\mathbb C}{\mathbb P}^1)))$ and
\begin{equation}\label{j1}
\Delta_{{\mathbb P}^1}\circ\alpha\,=\, {\rm Id}_{K^{\otimes 2}_{{\mathbb C}{\mathbb P}^1}},
\end{equation}
where $\alpha$ is the homomorphism in \eqref{ee2}.
So $\Delta_{{\mathbb P}^1}\circ\alpha$ is a holomorphic differential operator
of order three from $T{\mathbb C}{\mathbb P}^1$ to $K^{\otimes 2}_{{\mathbb C}{\mathbb P}^1}$ which is invariant
under the action of $\text{PSL}(2,{\mathbb C})$. The symbol of $\Delta_{{\mathbb P}^1}\circ\alpha$ is a holomorphic section
of $\text{Hom}(T{\mathbb C}{\mathbb P}^1,\, K^{\otimes 2}_{{\mathbb C}{\mathbb P}^1})\otimes (T{\mathbb C}{\mathbb P}^1)^{\otimes 3}
\,=\, {\mathcal O}_{{\mathbb C}{\mathbb P}^1}$. From \eqref{j1} it follows immediately that 
the symbol of $\Delta_{{\mathbb P}^1}\circ\alpha$ is the constant function $1$ on ${\mathbb C}{\mathbb P}^1$.

Consequently, given any holomorphic coordinate chart $(U,\, \phi)$ on $X$ compatible with the projective
structure $\delta$, we have a differential operator $TU\, \longrightarrow\, K^{\otimes 2}_U$ which is simply
the transport of $\Delta_{{\mathbb P}^1}$ using the biholomorphism $\phi$. Since $\Delta_{{\mathbb P}^1}$
is $\text{PSL}(2,{\mathbb C})$--invariant, these locally defined differential operators on $X$ 
patch together compatibly to produce a global third order differential operator
$TX\, \longrightarrow\, K^{\otimes 2}_X$. This differential operator coincides with $\Delta_\delta$ in
\eqref{delta}. Indeed, this follows immediately by comparing their expressions in terms of holomorphic
coordinate functions on $X$ compatible with the projective structure $\delta$.

So, $\Delta_\delta\, :\, J^3(X)\,\longrightarrow\, K^{\otimes 2}_X$ gives a splitting of the
jet sequence
\begin{equation}\label{jx}
0\,\longrightarrow\, K^{\otimes 2}_X \,\longrightarrow\,
J^3(TX)\,\longrightarrow\, J^2(TX)\,\longrightarrow\, 0
\end{equation}
on $X$. From this it follows immediately that the symbol of the differential operator $\Delta_\delta$ in
\eqref{delta} is the constant function $1$
(see also \cite[p.~273]{Hu}). Consequently, the sheaf of solutions of $\Delta_\delta$ is a complex local system
on $X$; let
\begin{equation}\label{lsx}
\Lambda_\delta\, \longrightarrow\, X
\end{equation}
be the local system given by the sheaf of solutions of $\Delta_\delta$
(the notation of \cite{Hu} is being used).
Since $\Delta_\delta$ is given by a splitting of \eqref{jx}, the holomorphic vector bundle underlying the local
system $\Lambda_\delta$ is identified with $J^2(TX)$.

Consider the holomorphic principal $\text{PSL}(2,{\mathbb C})$--bundle $P^0$ in \eqref{e3}. Let $D$ be the
holomorphic connection on $P^0$ corresponding to the projective structure $\delta$. The connection
on the adjoint bundle $\text{ad}(P^0)$ (for $P^0$) induced by $D$ will be denoted by $\widetilde D$. The
complex local system on $X$ given by the
sheaf of flat sections of $\text{ad}(P^0)$, for the connection $\widetilde D$, will be denoted by
\begin{equation}\label{j5}
\underline{\rm ad}(P^0).
\end{equation}

\begin{lemma}\label{lem2}
The two complex local systems $\Lambda_\delta$ (constructed in \eqref{lsx}) and
$\underline{\rm ad}(P^0)$ (in \eqref{j5}) on $X$ are canonically identified.
\end{lemma}

\begin{proof}
The kernel of the projection
$\Delta_{{\mathbb P}^1}$ in \eqref{ee3} is $J^2(T{\mathbb C}{\mathbb P}^1)$, because it is given by a
splitting of the exact sequence in \eqref{ee2}. In particular, the vector bundle corresponding to the
local system defined by the sheaf of solutions of $\Delta_{{\mathbb P}^1}$ is $J^2(T{\mathbb C}{\mathbb P}^1)$.
Recall that $\eta_2$ identifies $J^2(T{\mathbb C}{\mathbb P}^1)$
with ${\mathbb C}{\mathbb P}^1\times H^0({\mathbb C}{\mathbb P}^1,\, T{\mathbb C}{\mathbb P}^1)\,=\,
{\mathbb C}{\mathbb P}^1\times \text{sl}(2,{\mathbb C})$ (see \eqref{liso}). So the kernel of the projection
$\Delta_{{\mathbb P}^1}$ in \eqref{ee3} is identified with ${\mathbb C}{\mathbb P}^1\times \text{sl}(2,{\mathbb C})$.
This identification is evidently $\text{PSL}(2,{\mathbb C})$--equivariant; since $\Delta_{{\mathbb P}^1}$
is $\text{PSL}(2,{\mathbb C})$--invariant, the kernel of $\Delta_{{\mathbb P}^1}$ is equipped
with an action of $\text{PSL}(2,{\mathbb C})$. This action of $\text{PSL}(2,{\mathbb C})$ on
$\text{kernel}(\Delta_{{\mathbb P}^1})\,=\, {\mathbb C}{\mathbb P}^1\times \text{sl}(2,{\mathbb C})$
coincides with the diagonal action of $\text{PSL}(2,{\mathbb C})$.

Let $(U,\, \phi)$ be any coordinate chart on $X$ compatible with the projective structure $\delta$. Using
$\phi$, the kernel of $\Delta_{{\mathbb P}^1}$ produces a constant subsheaf of $TU$. Similarly,
the constant sheaf ${\mathbb C}{\mathbb P}^1\times \text{sl}(2,{\mathbb C})$ produces a constant sheaf
on $U$ with stalk $\text{sl}(2,{\mathbb C})$, using $\phi$. The above isomorphism of
the kernel of $\Delta_{{\mathbb P}^1}$ with ${\mathbb C}{\mathbb P}^1\times \text{sl}(2,{\mathbb C})$
produces an isomorphism of these two local systems on $U$.

The local system $\Lambda_\delta$ on $X$ is constructed by patching together these locally defined
constant subsheaves of $TU$, while
the local system $\underline{\rm ad}(P^0)$ is constructed by patching together these locally defined
constant sheaves with stalk $\text{sl}(2,{\mathbb C})$; the patchings are done using the transition functions
for the coordinate charts compatible with $\delta$. Now from the above observation that the kernel
$\Delta_{{\mathbb P}^1}$ is $\text{PSL}(2,{\mathbb C})$--equivariantly identified with ${\mathbb C}{\mathbb P}^1
\times \text{sl}(2,{\mathbb C})$ it follows immediately that these locally defined isomorphisms on $X$ produce
a global isomorphism between $\Lambda_\delta$ and $\underline{\rm ad}(P^0)$.
\end{proof}

Recall that the vector bundles on $X$ underlying the local systems $\Lambda_\delta$ and
$\underline{\rm ad}(P^0)$ are $J^2(TX)$ and ${\rm ad}(P^0)$ respectively. Therefore, Lemma \ref{lem2}
produces an isomorphism of $J^2(TX)$ with ${\rm ad}(P^0)$; to clarify, this isomorphism depends on the
projective structure $\delta$. Using this isomorphism between $J^2(TX)$ and ${\rm ad}(P^0)$, the bilinear
trace form on $\text{sl}(2,{\mathbb C})$ produces a fiberwise nondegenerate symmetric pairing
\begin{equation}\label{fd}
f_\delta\, :\, J^2(TX)\otimes J^2(TX)\, \longrightarrow\, {\mathcal O}_X\,
\end{equation}

Let
\begin{equation}\label{eb}
{\mathcal B}_\bullet \, :\, {\mathcal B}_0\, :=\, TX\, \stackrel{\Delta_\delta}{\longrightarrow}\,
{\mathcal B}_1\, :=\, K^{\otimes 2}_X
\end{equation}
be the two--term complex on $X$, where $\Delta_\delta$ is the differential operator in \eqref{delta}; to clarify,
the sheaf ${\mathcal B}_i$ in \eqref{eb} is at the $i$-th position of the complex.

The space of infinitesimal deformations of the pair $(X,\, \delta)$ are parametrized by the hypercohomology
${\mathbb H}^1({\mathcal B}_\bullet)$ \cite[p.~266]{Hu}. From the short exact sequence of sheaves
$$
0\,\longrightarrow\, \Lambda_\delta \,\longrightarrow\, TX\,\stackrel{\Delta_\delta}{\longrightarrow}\,
K^{\otimes 2}_X\,\longrightarrow\, 0
$$
(see \eqref{lsx}) we know that
\begin{equation}\label{c1}
{\mathbb H}^1({\mathcal B}_\bullet)\,=\, H^1(X,\, \Lambda_\delta)\, .
\end{equation}

Consider the short exact sequence of complexes
$$
\begin{matrix}
&& 0 && 0\\
&&\Big\downarrow && \Big\downarrow\\
&& 0 & \longrightarrow & K^{\otimes 2}_X\\
&& \Big\downarrow &&\,\,\,\,\,\, \Big\downarrow{\rm id}\\
{\mathcal B}_\bullet & : & TX& \stackrel{\Delta_\delta}{\longrightarrow} & K^{\otimes 2}_X\\
&&\,\,\,\,\,\, \Big\downarrow{\rm id} && \Big\downarrow\\
&& TX & \longrightarrow & 0\\
&&\Big\downarrow && \Big\downarrow\\
&& 0 && 0
\end{matrix}
$$
on $X$. It produces a short exact sequence of hypercohomologies
\begin{equation}\label{ehc}
0\,=\, H^0(X,\, TX) \,\longrightarrow\, H^0(X,\, K^{\otimes 2}_X) \,\stackrel{\alpha_1}{\longrightarrow}\,
{\mathbb H}^1({\mathcal B}_\bullet)
\end{equation}
$$
\stackrel{\alpha_2}{\longrightarrow}\, H^1(X,\, TX)\,\longrightarrow\,
H^1(X,\, K^{\otimes 2}_X)\,=\, 0\, .
$$
The homomorphism $\alpha_1$ in \eqref{ehc} corresponds to moving the projective structure keeping the
Riemann surface $X$ fixed \cite[p.~267\,(iii)]{Hu}, and the homomorphism $\alpha_2$ corresponds to the
forgetful map that sends an infinitesimal deformation of $(X,\, \delta)$ to the infinitesimal deformation of
$X$ corresponding to it obtained by simply forgetting the projective structure \cite[p.~266\,(ii)]{Hu}.

\subsection{A two--form on the moduli space of projective structures}

As before, $X$ is a smooth complex projective curve $X$ of genus $g\, \geq\, 2$, and $\delta$
is a projective structure on $X$. Consider the differential operator $\Delta_\delta$ in
\eqref{delta}. We will show that it satisfies a certain identity.

The natural contraction homomorphism
\begin{equation}\label{c2}
TX\otimes K^{\otimes 2}_X\, \longrightarrow\, K_X
\end{equation}
will be denoted by ``$\cdot$''. Let
\begin{equation}\label{p0}
p_0\, :\, J^3(TX)\, \longrightarrow\, TX
\end{equation}
be the natural projection.

The ${\mathcal O}_X$--linear homomorphism $\Delta_\delta\, :\, J^3(TX)\, \longrightarrow\, K^{\otimes 2}_X$
in \eqref{delta} produces an ${\mathcal O}_X$--linear homomorphism
\begin{equation}\label{H1}
H_1\, :\, \text{Sym}^2(J^3(TX))\, \longrightarrow\, K_X
\end{equation}
that sends any $(v\otimes w)+(w\otimes v)\, \in\, J^3(TX)_x\otimes J^3(TX)_x$, $x\, \in\, X$, to
$$p_0(v)\cdot \Delta_\delta(w)+ p_0(w)\cdot \Delta_\delta(v)\, \in\, (K_X)_x,$$ where $p_0$ is the
projection in \eqref{p0}, and ``$\cdot$'' is the pairing in \eqref{c2}.

We will construct another homomorphism $H_2$ from $\text{Sym}^2(J^3(TX)\otimes J^3(TX))$ to $K_X$. For that
we first recall that given vector bundles $V$ and $W$ on $X$, and $i\, \geq\, 0$, there are natural homomorphisms
\begin{equation}\label{fi}
f_1\, :\, J^{i+1}(V)\, \longrightarrow\, J^1(J^i(V))\ \ \text{ and }\ \ f_2\, :\, J^i(V)\otimes
J^i(W) \, \longrightarrow\, J^i(V\otimes W)\, .
\end{equation}
Now we have the following composition of homomorphisms
$$
\text{Sym}^2(J^3(TX))\,\hookrightarrow\, J^3(TX)\otimes J^3(TX)\, \stackrel{f_1\otimes f_1}{\longrightarrow}\,
J^1(J^2(TX))\otimes J^1(J^2(TX))
$$
\begin{equation}\label{dh2}
\stackrel{f_2}{\longrightarrow}\, J^1(J^2(TX)\otimes J^2(TX))
\, \stackrel{J^1(f_\delta)}{\longrightarrow}\, J^1({\mathcal O}_X)\, \stackrel{d}{\longrightarrow}\, K_X
\end{equation}
(see \eqref{fi} for $f_1\otimes f_1$ and \eqref{fd} for $J^1(f_\delta)$); note
that the de Rham differential $d\, :\, {\mathcal O}_X\, \longrightarrow\, K_X$,
being a first order differential operator, is given by an ${\mathcal O}_X$--linear
homomorphism $J^1({\mathcal O}_X)\, \longrightarrow\, K_X$. Let
\begin{equation}\label{H2}
H_2\, :\, \text{Sym}^2(J^3(TX))\, \longrightarrow\, K_X
\end{equation}
be the homomorphism given by the composition of homomorphisms in \eqref{dh2}.

\begin{lemma}\label{lem1}
The homomorphism $H_1$ in \eqref{H1} coincides with the homomorphism $H_2$ in \eqref{H2}.
\end{lemma}

\begin{proof}
In view of the explicit description of $\Delta_\delta$ in \eqref{dde}, this is a rather straightforward
computation. We will give an alternative proof of it using representations of $\text{PSL}(2, {\mathbb C})$.

It suffices to prove that
\begin{enumerate}
\item $H_1\,=\, H_2$ for $X\, =\, {\mathbb C}{\mathbb P}^1$, and

\item the identification $H_1\,=\, H_2$ on ${\mathbb C}{\mathbb P}^1$ is equivariant for the action of
$\text{PSL}(2, {\mathbb C})$ on ${\mathbb C}{\mathbb P}^1$.
\end{enumerate}
Indeed, in that case using coordinate charts on $X$, compatible with the projective structure $\delta$, we 
may transport the identification $H_1\,=\, H_2$ (for ${\mathbb C}{\mathbb P}^1$) to open subsets of $X$; the
equivariance condition in (2) then
ensures that these locally defined identifications over $X$ patch together compatibly to give a global
identification $H_1\,=\, H_2$ over $X$.

Since the homomorphism $\eta_3\circ(\eta_2)^{-1}\, :\, J^2(T{\mathbb C}{\mathbb P}^1)
\,\longrightarrow\, J^3(T{\mathbb C}{\mathbb P}^1)$ (see \eqref{ee}) gives a $\text{PSL}(2, {\mathbb C})$--equivariant
splitting of the short exact sequence in \eqref{ee2}, we have
$$
J^3(T{\mathbb C}{\mathbb P}^1)\,=\, J^2(T{\mathbb C}{\mathbb P}^1)\oplus K^{\otimes 2}_{{\mathbb C}{\mathbb P}^1}\ .
$$
This implies that
\begin{equation}\label{jd}
\text{Sym}^2(J^3(T{\mathbb C}{\mathbb P}^1))\,=\, \text{Sym}^2(J^2(T{\mathbb C}{\mathbb P}^1)) \oplus
K^{\otimes 4}_{{\mathbb C}{\mathbb P}^1}\oplus (K^{\otimes 2}_{{\mathbb C}{\mathbb P}^1}\otimes
J^2(T{\mathbb C}{\mathbb P}^1)),
\end{equation}
and this isomorphism commutes with the actions of $\text{PSL}(2, {\mathbb C})$.
The vector bundle $\text{Sym}^2(J^2(T{\mathbb C}{\mathbb P}^1))$ is trivial, because
$J^2(T{\mathbb C}{\mathbb P}^1)$ is trivialized by $\eta_2$ in \eqref{ee}. Therefore,
\begin{equation}\label{u1}
H^0({\mathbb C}{\mathbb P}^1, \,\, \,\text{Hom}(\text{Sym}^2(J^2(T{\mathbb C}{\mathbb P}^1)),\, 
K_{{\mathbb C}{\mathbb P}^1}))\,=\, 0
\end{equation}
as the degree of $K_{{\mathbb C}{\mathbb P}^1}$ is negative. Also, we have
$$
H^0({\mathbb C}{\mathbb P}^1, \,\, \text{Hom}(K^{\otimes 4}_{{\mathbb C}{\mathbb P}^1},\, 
K_{{\mathbb C}{\mathbb P}^1}))\,=\, H^0({\mathbb C}{\mathbb P}^1, \,
(T{{\mathbb C}{\mathbb P}^1})^{\otimes 3}).
$$
This implies that
$$
H^0({\mathbb C}{\mathbb P}^1, \,\, \text{Hom}(K^{\otimes 4}_{{\mathbb C}{\mathbb P}^1},\,
K_{{\mathbb C}{\mathbb P}^1}))
\,=\, \text{Sym}^6({\mathbb C}^2)
$$
as a $\text{PSL}(2, {\mathbb C})$--module. Now taking invariants for the action of $\text{PSL}(2, {\mathbb C})$,
\begin{equation}\label{u2}
H^0({\mathbb C}{\mathbb P}^1, \,\, \text{Hom}(K^{\otimes 4}_{{\mathbb C}{\mathbb P}^1},\, 
K_{{\mathbb C}{\mathbb P}^1}))^{\text{PSL}(2, {\mathbb C})}\,=\,0.
\end{equation}

Next, we have
$$
H^0({\mathbb C}{\mathbb P}^1,\,\,\, \text{Hom}(K^{\otimes 2}_{{\mathbb C}{\mathbb P}^1}\otimes
J^2(T{\mathbb C}{\mathbb P}^1),\, K_{{\mathbb C}{\mathbb P}^1}))
$$
$$
=\,
H^0(X,\, J^2(T{\mathbb C}{\mathbb P}^1))^*\otimes H^0(X,\, \text{Hom}(K^{\otimes 2}_{{\mathbb C}{\mathbb P}^1},\,
K_{{\mathbb C}{\mathbb P}^1}))
$$
because $J^2(T{\mathbb C}{\mathbb P}^1)$ is the trivial bundle ${\mathbb C}{\mathbb P}^1\times H^0(X,\,
J^2(T{\mathbb C}{\mathbb P}^1))$ (trivialized by $\eta_2$ in \eqref{ee}). Consequently, we have
$$
H^0({\mathbb C}{\mathbb P}^1,\, \text{Hom}(K^{\otimes 2}_{{\mathbb C}{\mathbb P}^1}\otimes
J^2(T{\mathbb C}{\mathbb P}^1),\, K_{{\mathbb C}{\mathbb P}^1}))\,=\,
H^0(X,\, J^2(T{\mathbb C}{\mathbb P}^1))^*\otimes H^0(X,\, T{\mathbb C}{\mathbb P}^1)
$$
$$
\,=\,\text{sl}(2,{\mathbb C})^*\otimes \text{sl}(2,{\mathbb C}) \,=\,
\text{sl}(2,{\mathbb C})^*\otimes \text{sl}(2,{\mathbb C})^*
$$
as a $\text{PSL}(2, {\mathbb C})$--module. Now taking invariants for the action of $\text{PSL}(2, {\mathbb C})$,
\begin{equation}\label{u3}
H^0({\mathbb C}{\mathbb P}^1,\, 
\text{Hom}(K^{\otimes 2}_{{\mathbb C}{\mathbb P}^1}\otimes
J^2(T{\mathbb C}{\mathbb P}^1),\, K_{{\mathbb C}{\mathbb P}^1}))^{\text{PSL}(2, {\mathbb C})}
\,=\,{\mathbb C}.
\end{equation}
Note that $1\, \in\, (\text{sl}(2,{\mathbb C})^*\otimes \text{sl}(2,{\mathbb C})^*)^{\text{PSL}(2, {\mathbb C})}$
is given by the map $\text{sl}(2,{\mathbb C})\otimes \text{sl}(2,{\mathbb C})\, \longrightarrow\, \mathbb C$
defined by $A\otimes B\, \longrightarrow\, \text{trace}(AB)$.

In view of \eqref{u1}, \eqref{u2} and \eqref{u3}, from \eqref{jd} we conclude that
$$
H^0({\mathbb C}{\mathbb P}^1,\, \, \text{Hom}(\text{Sym}^2(J^3(T{\mathbb C}{\mathbb P}^1)),\,
K_X))^{\text{PSL}(2, {\mathbb C})} \,=\, {\mathbb C}.
$$
In other words, up to multiplication by a scalar, there is a unique nonzero $\text{PSL}(2, {\mathbb
C})$--equivariant homomorphism from $\text{Sym}^2(J^3(T{\mathbb C}{\mathbb P}^1))$ to $K_X$. Hence the
homomorphisms $H_1$ and $H_2$ coincide up to multiplication by a scalar when $X\,=\, {\mathbb C}{\mathbb P}^1$.
It is now easy to see that $H_1\,=\, H_2$ when $X\,=\, {\mathbb C}{\mathbb P}^1$.
\end{proof}

Take two distinct point $a,\, b\, \in\, X$. Let $U_a\, :=\, X\setminus\{b\}$ and $U_b\, :=\, X\setminus\{a\}$ 
be the open subsets; define $U\, :=\, U_a\cap U_b$. The $1$-cocycles for ${\mathcal B}_\bullet$ consists of
$$
(\theta,\, \omega_a,\, \omega_b) \, \in\, H^0(U,\, TU)\oplus H^0(U_a,\, K^{\otimes 2}_{U_a})
\oplus H^0(U_b,\, K^{\otimes 2}_{U_b})
$$
such that $\omega_a-\omega_b\, =\, \Delta_\delta(\theta)$, where $\Delta_\delta$ is the differential
operator in \eqref{delta}. A cocycle $(\theta,\, \omega_a,\, \omega_b)$ is
a $1$-coboundary if there is an element
$$
(\theta_a,\, \theta_b)\, \in\, H^0(U_b,\, TU_a)\oplus H^0(U_b,\, TU_b)
$$
such that $\theta\, =\, \theta_a-\theta_b$, $\omega_a\,=\, \Delta_\delta(\theta_a)$ and
$\omega_b\,=\, \Delta_\delta(\theta_b)$. The hypercohomology ${\mathbb H}^1({\mathcal B}_\bullet)$ (see \eqref{eb})
is identified with the quotient of $1$-cocycles by the $1$-coboundaries.

The residue of any $\alpha\, \in\, H^0(U,\, K_{U})$ at the point $a\,\in\, X$
will be denoted by $\text{Res}_a(\alpha)$.

Take two $1$-cocycles $\mathbf{c}^1\, :=\, (\theta^1,\, \omega^1_a,\, \omega^1_b)$ and
$\mathbf{c}^2\, :=\, (\theta^2,\, \omega^2_a,\, \omega^2_b)$ for ${\mathcal B}_\bullet$ in \eqref{eb}.
So we have $$\theta^1\cdot \omega^2_b,\,\,\, \theta^2\cdot \omega^1_a\,\, \in\,\, H^0(U,\, K_{U})\, .$$
Define
\begin{equation}\label{ttd}
\widehat{\Theta}'(\mathbf{c}^1,\, \mathbf{c}^2)\,:=\, \text{Res}_a(\theta^1\cdot \omega^2_b+\theta^2\cdot \omega^1_a)
\, \in\, \mathbb C,
\end{equation}
where ``$\cdot$'' is the natural pairing $TX\otimes K^{\otimes 2}_X\,\longrightarrow\, K_X$.

\begin{theorem}\label{thm1}
If $\mathbf{c}^1$ in \eqref{ttd} is a $1$-coboundary, then $\widehat{\Theta}'(\mathbf{c}^1,\, \mathbf{c}^2)\,=\,0$.
\end{theorem}

\begin{proof}
Since $\mathbf{c}^1\, :=\, (\theta^1,\, \omega^1_a,\, \omega^1_b)$ is a $1$-coboundary, there is
$$
(\theta_a,\, \theta_b)\, \in\, H^0(U_b,\, TU_a)\oplus H^0(U_b,\, TU_b)
$$
such that $\theta^1\, =\, \theta_a-\theta_b$, $\omega^1_a\,=\, \Delta_\delta(\theta_a)$ and $\omega^1_b\,=\,
\Delta_\delta(\theta_b)$. So $\theta^1\cdot \omega^2_b+\theta^2\cdot \omega^1_a$ in \eqref{ttd} has the
expression
$$
\theta^1\cdot \omega^2_b+\theta^2\cdot \omega^1_a\,=\, (\theta_a-\theta_b)\cdot \omega^2_b+
\theta^2\cdot \Delta_\delta(\theta_a)\,=\,\theta_a\cdot \omega^2_b +
\theta^2\cdot \Delta_\delta(\theta_a) - \theta_b\cdot \omega^2_b
$$
\begin{equation}\label{tz1}
=\, \theta_a\cdot (\omega^2_a + \Delta_\delta(\theta^2))+
\theta^2\cdot \Delta_\delta(\theta_a) - \theta_b\cdot \omega^2_b\,=\,
\theta_a\cdot \omega^2_a - \theta_b\cdot \omega^2_b +H_1(\theta_a,\, \theta^2)\, ,
\end{equation}
where $H_1$ is the homomorphism in \eqref{H1}. Since the residues of
$\theta_a\cdot \omega^2_a$ and $\theta_b\cdot \omega^2_b$ vanish at the point $a$ (and at $b$),
\begin{equation}\label{tz2}
\widehat{\Theta}'(\mathbf{c}^1,\, \mathbf{c}^2)\,:=\, \text{Res}_a(H_1(\theta_a,\, \theta^2))
\end{equation}
(see \eqref{ttd} and \eqref{tz1}). From Lemma \ref{lem1},
\begin{equation}\label{pc}
\text{Res}_a(H_1(\theta_a,\, \theta^2))\,=\, \text{Res}_a(H_2(\theta_a,\, \theta^2))\, ;
\end{equation}
on the other hand, $H_2(\theta_a,\, \theta^2)$ is exact on $U\, :=\, U_a\cap U_b$ (the
final homomorphism in \eqref{dh2} is the de Rham differential). Hence we have
$$
\text{Res}_a(H_2(\theta_a,\, \theta^2))\,=\, 0\, .
$$
Now from \eqref{tz2} we know that $\widehat{\Theta}'(\mathbf{c}^1,\, \mathbf{c}^2)\,=\,0$.
\end{proof}

Theorem \ref{thm1} has the following immediate consequence.

\begin{corollary}\label{cor1}
The pairing $\widehat{\Theta}'$ in \eqref{ttd} produces a alternating form
$$
\widehat{\Theta}(\delta)\, :\,\,\, \bigwedge\nolimits^2 {\mathbb H}^1({\mathcal B}_\bullet)\,
\longrightarrow\,\mathbb C\, .
$$
\end{corollary}

Note that $\widehat{\Theta}$ in Corollary \ref{cor1} produces a $2$--form on the moduli
space ${\mathcal P}_g$ in \eqref{v2}. From
its construction it is evident that $\widehat{\Theta}$ is an algebraic form on ${\mathcal P}_g$.

{}From the construction of the $2$--form $\widehat{\Theta}$ on ${\mathcal P}_g$ it is straightforward
to deduce that it coincides with the symplectic form $\Theta_g$ in \eqref{e12}.


\end{document}